\documentclass[12pt]{amsart}
\usepackage[pdftex,colorlinks,citecolor=blue,urlcolor=blue]{hyperref}
\usepackage{amssymb}
\usepackage{graphicx,color}
\usepackage{amsmath}
\usepackage{comment}
\usepackage{MnSymbol}
\usepackage{cleveref}
\usepackage{bbold}
\numberwithin{equation}{section}

\newtheorem{theorem}{Theorem}[section]

\newtheorem{corollary}[theorem]{Corollary}

\newtheorem{lemma}[theorem]{Lemma}

\theoremstyle{definition}

\newtheorem{remark}[theorem]{Remark}

\newcommand{\mrank}{\mu\text{rank}}

\newcommand{\R}{{\mathbb R\strut}}

\title{Minimality of tensors of fixed multilinear rank}

\author[A. Heaton]{Alexander Heaton}
\email{alexheaton2@gmail.com \vspace{-10pt}}
\address{\scriptsize Max Planck Institute for Mathematics in the Sciences, Leipzig, Technische Universit\"{a}t Berlin}
\author[K. Kozhasov]{Khazhgali Kozhasov}
\email{k.kozhasov@tu-braunschweig.de \vspace{-10pt}}
\address{
     \scriptsize Technische Universit\"{a}t Braunschweig, Institut f\"{u}r Analysis und Algebra
}
\author[L. Venturello]{Lorenzo venturello}
\email{lorenzo.venturello@mis.mpg.de, lorenzo.venturello@hotmail.it \vspace{-10pt}}
\address{
	 \scriptsize Max Planck Institute for Mathematics in the Sciences, Leipzig
}

\begin{document}
	\begin{abstract}
    We discover a geometric property of the space of tensors of fixed multilinear (Tucker) rank. Namely, it is shown that real tensors of fixed multilinear rank form a minimal submanifold of the Euclidean space of tensors endowed with the Frobenius inner product. We also establish the absence of local extrema for linear functionals restricted to the submanifold of rank-one tensors, finding application in statistics. 
\end{abstract}
\maketitle
\section{Introduction}\label{section:introduction}
In the following by $\R^{n_1}\otimes \dots\otimes \R^{n_d}$ we denote the space of real $(n_1,\dots,n_d)$-tensors whose elements are identified with arrays $T=(t_{i_1\dots i_d})$ of real numbers. The space $\R^{n_1}\otimes \dots\otimes \R^{n_d}$ is endowed with the standard \emph{Frobenius inner product} that is defined by
\begin{align}\label{Frob}
    	\langle T, S \rangle ~ = ~ \sum_{j=1}^d\sum_{i_j=1}^{n_j} t_{i_1\dots i_d}s_{i_1\dots i_d},
	\end{align}
	    where $T=(t_{i_1\dots i_d})$, $ S=(s_{i_1\dots i_d})$. The \emph{multilinear rank} of $T \in \R^{n_1} \otimes \cdots \otimes \R^{n_d}$ is the tuple 
		\[
			\mrank(T)~ =~\min\left\{(r_1,\dots,r_d)\in \mathbb{N}_0^d: ~T\in W_1\otimes\cdots\otimes W_d, ~W_j\subseteq \R^{n_j},~ \dim W_j = r_j\right\},
		\]
		where $\mathbb{N}_0 = \{0,1,2,\dots\}$ and each $W_j$ is a linear subspace of $\R^{n_j}$.
The number $r_j$ is equal to the rank of the $n_j \times \prod_{k \neq j} n_k$ matrix obtained by flattening or unfolding or matricizing the tensor along mode $j$ \cite[Theorem $6.13$]{Hackbusch}. In the case $d=2$ the multilinear rank $(r_1,r_2)$ of a matrix $T$ satisfies $r_1=r_2=\text{rank}(T)$, since the rank of the row and column space of any matrix coincide. For tensors with more than $2$ modes, the multilinear rank is different than the classical rank
	\begin{equation*}
	    \text{rank}(T) = \text{min} \left\{ r \in \mathbb{N}_0 ~:~ T = \sum_{i=1}^r v_{1,i} \otimes \cdots \otimes v_{d,i}, \,\,\,\,\, v_{j,i} \in \R^{n_j} \right\}.
	\end{equation*}
	The latter notion of rank is also important for many applications (see \cite{Hitchcock1927ExpressionTensorPolyadicSumOfProducts}), but will not be the focus of this article. 
For $d\geq 1$ and $\mathbf{r}\in \mathbb{N}_0^d$, let $\mathcal{T}_{d,\mathbf{r}}$ be the set of order $d$ tensors with multilinear rank $\mathbf{r}$. 
It is well-known (see, for example, \cite{UschmajewVandereycken2013geometryAlgorithmsUsingHierarchicalTensors}) that this set is a smooth submanifold of $\R^{n_1}\otimes\cdots\otimes \R^{n_d}$ of dimension
		\[
			\dim(\mathcal{T}_{d,\mathbf{r}})~ = ~ \sum_{j=1}^d r_j(n_j-r_j) + \prod_{j=1}^d r_j.
		\]
	The topology of $\mathcal{T}_{d,\mathbf{r}}$ was recently studied in \cite{ComonLimQiYe2020TopologyOfTensorRanks}. They show, for example, that $\mathcal{T}_{d,\mathbf{r}}$ is path-connected unless $n_j = r_j = \prod_{k \neq j} r_j$ for some $j$. The notion of multilinear rank was introduced in \cite{Kruskal1989RankDecompositionUniqueness3WayNwayArrays} and popularized in \cite{DeLathauwerDeMoorVandewalle2000aMultilinearSingularValueDecomposition} where it was used to show that the Tucker decomposition (see \cite{Tucker1966SomeMathematicalNotesOnThreeModeFactorAnalysis}) provides a convincing generalization of the matrix singular value decomposition (SVD). The set $\mathcal{T}_{d,\mathbf{r}}$ and the Tucker decomposition in general have been used in numerous interesting applications, including the technique TensorFaces in computer vision \cite{VasilescuTerzopoulos2002MultilinearAnalysisOfImageEnsembles}. Note that throughout this article we refer to the set of tensors of fixed multilinear rank, rather than the related \textit{subspace variety} consisting of all tensors with $\mu\text{rank}(T) \leq (r_1,\dots,r_d)$, which is also a well-studied object of interest \cite[Chapter 7]{Landsberg2012TensorsGeometryApplicationsTEXT}.
	
	A common task is to find the best low-rank approximation of a tensor \cite{DeLathauwerDeMoorVandewalle2000BestRank1AndRankR1R2-RNApproximationHigherOrderTensors}. For matrices, the celebrated Eckart-Young theorem \cite{Eckart1936TheAO} states that truncating the SVD yields a closed-form solution,
	but for tensor rank the problem is ill-posed \cite{deSilvaLim2008TensorRankAndIllposednessBestLowRankApprox} due to the phenomenon of \textit{border rank}. 
	The subspace variety is Zariski closed, providing a nice setting for existence, but is not smooth. As a consequence, computing the best multilinear approximation over the subspace variety becomes more difficult, since not all points admit tangent spaces. 
	
	In contrast, the problem of finding the best multilinear low-rank approximation is well-posed on $\mathcal{T}_{d,\mathbf{r}}$ with a unique solution with respect to the Frobenius norm \cite[Corollary 4.5]{deSilvaLim2008TensorRankAndIllposednessBestLowRankApprox}. The manifold structure of low-rank tensors has been used in numerical analysis and computational physics (see \cite{GrasedyckKressnerTobler2013LiteratureSurveyLowRankTensorApproximation} for an overview) and, in general, methods of Riemannian optimization can be utilized \cite{AbsilMahonySepulchre2008OptimizationAlgorithmsMatrixManifoldsTEXT, EdelmanAriasSmith1999GeometryAlgorithmsOrthogonalityConstraints}. In recent work \cite{KressnerSteinlechnerVandereycken2014LowRankTensorCompletionRiemannianOptimization}, the authors studied the problem of tensor completion, filling in the missing entries of a tensor to achieve a low-rank tensor. Tensors of fixed multilinear rank $\mathcal{T}_{d,\mathbf{r}}$ were used due to their manifold structure, implementing a version of nonlinear conjugate gradient method, see also \cite{DaSilvaHerrmann2015OptimizationOnHierarchicalTuckerManifoldTensorCompletion}. $\mathcal{T}_{d,\mathbf{r}}$ was also used in \cite{EldenSavas2009NewtonGrassmannBestMultilinearRankApproximationTensor, SavasLim2010QuasiNewtonMethodsGrassmanniansMultilinearApproxTensors} to formulate the tensor approximation problem over a product of Grassmann manifolds. In \cite{KochLubich2010DynamicalTensorApproximation}, the manifold structure of $\mathcal{T}_{d,\mathbf{r}}$ was used to derive nonlinear differential equations whose solution is a time-varying family of tensors $S(t) \in \mathcal{T}_{d,\mathbf{r}}$ which provide the best multilinear rank approximation to a given family $T(t) \in \R^{n_1} \otimes \cdots \otimes \R^{n_d}$, a problem called \textit{dynamical tensor approximation}. Therefore the geometry of the set $\mathcal{T}_{d,\mathbf{r}}$ is important for applications.

Our main result discovers a new geometric property of $\mathcal{T}_{d,\mathbf{r}}$, namely its minimality. 
\begin{theorem}\label{main}
For any $\mathbf{r}\in \mathbb{N}_0^d$ the manifold $\mathcal{T}_{d,\mathbf{r}}$ is a minimal submanifold of the Euclidean space $(\R^{n_1}\otimes\dots\otimes \R^{n_d},\langle\cdot,\cdot\rangle)$, that is, its mean curvature vector field is identically zero. 
\end{theorem}
Minimal submanifolds are mathematical models of soap films: they minimize the volume locally around every point. When $d=2$ the manifold $\mathcal{T}_{2,\mathbf{r}}$, $\mathbf{r}=(r,r)$, consists of $n_1\times n_2$ matrices of rank $r$. Its minimality was recently proved in \cite{bordemann2020minimality, kozhasov2020minimality} and for $n_1=n_2=r+1$ earlier in \cite{Tka}. Thus Theorem \ref{main} generalizes this property from matrices to higher order tensors.

The case $\mathbf{r}=(1,\dots,1)$ is of particular interest, since the elements of $\mathcal{T}_{d,(1,\dots,1)}$ are precisely tensors of rank 1. The manifold $\mathcal{T}_{d,(1,\dots,1)}$ is the (affine) cone over the real part of the (projective) \emph{Segre variety}. In what follows we use the term ``Segre variety" when referring to the manifold $\mathcal{T}_{d,(1,\dots,1)}$ and denote it for simplicity by $\mathcal{S}_d$. If we slice $\mathcal{T}_{d,(1,\dots,1)}$ with the affine hyperplane $A$ containing the tensors with sum of coordinates equal to 1 and consider its nonnegative part, we obtain a \emph{statistical model} $\mathcal{M}$ which parametrizes $d$-tuples of independent discrete random variables $(X_1,\dots,X_d)$. The tensor $T=(t_{i_1\dots i_d})$ expresses the joint probability $t_{i_1\dots i_d}=\mathbb{P}(X_1=i_1,\dots,X_d=i_d)$, and the condition of having rank 1 corresponds to statistical independence of $X_1,\dots, X_d$. The popular Wasserstein distance between two probability distributions arises from optimal transport, measuring how much work is required to ``move'' one distribution to the other. In this discrete setting, the Wasserstein distance between a distribution $T\in A$ and $\mathcal{M}$ is the smallest scaling factor for which a certain polyhedron centered at $T$ intersects $\mathcal{M}$. The closest points on $\mathcal{M}$ to $T$ are those in the intersection. 
In \cite[Section 6]{wasserstein}~it was  experimentally observed that the optimal solution to this problem is never attained in the relative interior of a face of maximal dimension. One possible explanation for this experimental observation would be that the restriction of a linear functional to $\mathcal{M}$ attains no local minima (respectively maxima) in its relative interior. Corollary \ref{cor:M} below shows that it is indeed the case. This in turn follows from the following related result for the Segre variety $\mathcal{S}_d=\mathcal{T}_{d,(1,\dots,1)}$. In the statement of this theorem, $\vec{\ell}$ is a vector in $\R^{n_1}\otimes\cdots\otimes\R^{n_d}$ representing a linear functional.

\begin{theorem}\label{Theorem:Linear}
Let $T\in \mathcal{S}_{d}$ and let $P_{\vec{\ell}}=\{T+V\,:\,\langle \vec{\ell},V\rangle = 0\}\subset \R^{n_1}\otimes\cdots\otimes\R^{n_d}$ be an affine hyperplane containing the tangent space to $\mathcal{S}_{d}$ at $T$. Then no neighborhood of $T$ in $\mathcal{S}_{d}$ is completely contained in one of the half spaces $P_{\vec{\ell}}^\pm = \{ T+V\,:\,\pm\langle \vec{\ell},V\rangle\geq 0\}$. In particular, the linear functional $L_{\vec{\ell}}(S) =\langle \vec{\ell},S\rangle$ does not attain a local maximum or minimum on $\mathcal{S}_{d}$. 
\end{theorem}

A minimal surface in $\R^3$ with non-degenerate second fundamental form has saddle-like shape locally around every point. 
In particular, no linear functional can be minimized or maximized on such a surface. 
An analogous fact is true for minimal submanifolds in $\R^n$ satisfying some non-degeneracy conditions, see Remark \ref{rem:non-degenerate}. 
In view of this and the minimality of $\mathcal{T}_{d,\mathbf{r}}$ (see Theorem \ref{main}), the assertion of Theorem \ref{Theorem:Linear} is anticipated. However to prove it, in Section \ref{sec:Linear} we have to deal with finer geometric information of $\mathcal{S}_d$ than just its mean curvature. 
\begin{corollary}\label{cor:A}
    Let $A$ be an affine hyperplane in $\R^{n_1}\otimes\dots\otimes\R^{n_d}$, and let $L:\R^{n_1}\otimes\dots\otimes\R^{n_d}\to \mathbb{R}$ be a linear functional not constant on $A$. \color{black} Then the restriction of $L$ to $\mathcal{S}_d\cap A$ has no local minima (respectively maxima).
\end{corollary}
If we take $A$ to be the affine hyperplane of $\R^{n_1}\otimes\dots\otimes\R^{n_d}$ consisting of tensors with entries summing up to $1$, we  obtain the following application in statistics.
\begin{corollary}\label{cor:M}
    Let $\mathcal{M}$ be the statistical model of $d$ independent discrete random variables. Then any linear functional which is not constant on $\mathcal{M}$ has no local minimum (or maximum) on the relative interior of~$\mathcal{M}$. 
\end{corollary}

\color{black}



\section{Preliminaries}\label{section:preliminaries}
	
This section reviews basic concepts we will need. For more on tensors and tensor decompositions, see \cite{Landsberg2012TensorsGeometryApplicationsTEXT}, while for differential geometry see \cite{Chavel2006RiemannianGeometry, KN1969}. We first recall a definition of the mean curvature vector field of a submanifold of an Euclidean space, see \cite[Ch.\text{\MakeUppercase{ \romannumeral 7}}]{KN1969} for more details. 
	
Let $(\R^n,\langle\cdot,\cdot\rangle)$ be an Euclidean space and let $M\subset \R^n$ be a smooth $m$-dimensional submanifold.  Consider a local parametrization $r: U\rightarrow M$ of $M$, where $U$ is an open subset of $\R^m$. The first order partial derivatives $\partial_{u_1}r(u),\dots,\partial_{u_m} r(u)$ form a basis of the tangent space $T_{r(u)} M$ of $M$ at $r(u)$ and we denote by $G_{r(u)}=\left(\langle \partial_{u_i} r(u),\partial_{u_j} r(u)\rangle\right)$ the \emph{Gram matrix} of the metric of $M$ with respect to this basis. Thus, $G_{r(u)}$, $u\in U$, is a smooth field of positive definite matrices along $r(U)\subset M$.  
\emph{The second fundamental form of $M$} is a symmetric bilinear form $b$ on the tangent bundle $TM$ to $M$ with values in the normal bundle $(TM)^{\perp}\subset \R^n$ to $M$ defined at a point $r(u)\in M$ via
\begin{align}\label{defn:b-the-normal-vector-valued-bilinear-form}
b(u,v) = \sum_{i,j=1}^m u^iv^j\left(\partial^{\,2}_{u_i u_j}r(u)\right)^\perp,
\end{align}
where $u=\sum_{i=1}^m u^i\partial_{u_i}r(u)$, $v=\sum_{j=1}^m v^j\partial_{u_j} r(u)\in T_{r(u)}M$ are arbitrary tangent vectors and $\left(\partial^{\,2}_{u_i u_j} r(u)\right)^\perp \in (T_{r(u)}M)^\perp$ is the \emph{normal component} of the vector $\partial^{\,2}_{u_i u_j} r(u)\in \R^n$. The \emph{mean curvature vector} of $M$ at $r(u)\in M$ is defined to be
\begin{align}\label{mean_1}
    H_{r(u)} = \sum_{i,j=1}^m \left(G^{-1}_{r(u)}\right)_{ij}  b\left(\partial_{u_i} r(u),\partial_{u_j} r(u)\right) = \sum_{i,j=1}^m \left(G^{-1}_{r(u)}\right)_{ij} \left(\partial^{\,2}_{u_iu_j} r(u)\right)^\perp, 
\end{align}
where $G^{-1}_{r(u)}$ denotes the inverse to the positive definite matrix $G_{r(u)}$. The definition depends only on the embedding $M\subset V$ and not on the choice of the local parametrization. The smooth field $H_{r(u)}$, $u\in U$, of normal vectors is the \emph{mean curvature vector field} of $M$ along the open set $r(U)\subset M$. By gluing ``local'' definitions of the mean curvature vector field along open sets $r(U)$ from an open cover of $M$ one obtains the smooth field of normal vectors called the \emph{mean curvature vector field} of $M$. It will be convenient to adapt the following ``formal'' writing of \eqref{mean_1}
\begin{align}\label{eq: H_formal}
H_{r(u)} = \mathrm{Tr}\left (G^{-1}_{r(u)}\cdot \left(\mathrm{d}^{\,2} r(u)\right)^\perp\right),    
\end{align}
where $\mathrm{d}^{\,2} r(u)$ is the symmetric matrix whose $(i,j)$-th entry is the vector $\partial^{\,2}_{u_i u_j} r(u)$ and $(\cdot )^\perp$ is applied entry-wise. 

A submanifold $M\subset \R^n$ of the Euclidean space $(\R^n,\langle\cdot,\cdot,\rangle)$ is called \emph{minimal} if the mean curvature vector field vanishes. Thus, in order to prove minimality of $M\subset \R^n$ one can show that for each point $x\in M$ there is a local parametrization $r: U\rightarrow M$ with $x=r(u)$ and such that $H_{r(u)}=0$. 

Next, we state the definition of multilinear rank of a tensor in terms of ranks of its flattening matrices.

For a tensor $T=(t_{i_1\dots i_d}) \in \R^{n_1}\otimes \cdots\otimes \R^{n_d}$ and an index $j\in [d]=\{1,\dots,d\}$ a \emph{$j$-th flattening} of $T$ is the $n_j\times \prod_{k\neq j} n_k$ matrix whose rows are indexed by $i_j\in [n_j]$, columns are indexed by \\ $(d-1)$-tuples~$(i_1,\dots,i_{j-1},i_{j+1},\dots, i_d)\in [n_1]\times \cdots \times [n_{j-1}]\times [n_{j+1}]\times \cdots\times [n_d]$ (ordered in an arbitrary but a priori fixed way) and its $(i_j,(i_1,\dots,i_{j-1}, i_{j+1},\dots, i_d))$-th entry equals $t_{i_1\dots i_{j-1} i_j i_{j+1}\dots i_d}$. It is well known that $T$ has multilinear rank $\mathbf{r}=(r_1,\dots,r_d)$ if and only if for $j\in [d]$ the $j$-th flattening matrix of $T$ has rank $r_j$, see \cite[Theorem $6.13$]{Hackbusch}. 
Finally, note that the dot products between the rows of the $j$-th flattening of $T$ with indices $\lambda, \lambda'\in [n_j]$ equals
\begin{align}\label{eq:flat}
    \sum_{k\in [d]\setminus j} \sum_{i_k=1}^{n_k} t_{i_1\dots i_{j-1} \lambda i_{j+1}\dots i_d}\,t_{i_1\dots i_{j-1} \lambda' i_{j+1}\dots i_d}.
\end{align}
\section{Proof of main theorem}\label{section:main-theorem}

In this section we prove the main Theorem \ref{main}. The proof outline is as follows: First, we use a transitive action of a product of orthogonal groups to give a local parametrization of the manifold $\mathcal{T}_{d,\mathbf{r}}$ at a point. We then use this parametrization to compute the Gram matrix, and observe its block-diagonal structure with blocks that are themselves Gram matrices of the rows of the various flattenings of the tensor. After using the block structure to show the vanishing of most terms contributing to the mean curvature, we observe that the only remaining terms are zero upon projecting to the normal space, completing the proof.\\

		\textit{Proof of Theorem \ref{main}.} The product of orthogonal groups $O(\mathbf{n})=O(n_1)\times \cdots \times O(n_d)$ acts on $\R^{n_1}\otimes \cdots \otimes \R^{n_d}$. More explicitly, $\mathbf{g}=(g^1,\dots,g^d)\in O(\mathbf{n})$ acts on the tensor $T=(t_{i_1\dots i_d})\in \R^{n_1}\otimes \dots\otimes \R^{n_d}$ as
		\begin{align}\label{eq:action}
            \mathbf{g}^*T=\left(\sum_{j=1}^d \sum_{k_j=1}^{n_j} g^1_{i_1k_1}\dots g^d_{i_dk_d} t_{k_1\dots k_d}\right),
		\end{align}
		preserving the Frobenius inner product \eqref{Frob}. This action also preserves the multilinear rank and hence restricts to an action on manifolds $\mathcal{T}_{d,\mathbf{r}}$. In the following, we fix an orthonormal basis $\{e^j_1,\dots,e^j_{n_j}\}$ of $\R^{n_j}$. In view of the above we have $H_T=0$ if and only if $H_{\mathbf{g}^*T} = 0$ for any $\mathbf{g}\in O(\mathbf{n})$. In particular, since the orthogonal group $O(n_j)$ acts transitively on the Grassmannian $Gr(r_j,n_j)$ of $r_j$-planes in $\R^{n_j}$, to prove that the mean curvature vector $H_T$ at $T\in \mathcal{T}_{d,\mathbf{r}}$ is zero we can first assume that $T\in W_1\otimes\dots\otimes W_d$ with $W_j=\mathrm{Span}\{e^j_1,\dots,e^j_{r_j}\}$, \[
			T ~ = ~\sum_{j=1}^d\sum_{i_j=1}^{r_j} t_{i_1\dots i_d} ~e^1_{i_1}\otimes \cdots \otimes e^d_{i_d}.
		\]
		For $1\leq \alpha,\beta\leq n_j$ let $E^j_{\alpha\beta}$ be the $(\alpha,\beta)$-th matrix unit of size $n_j\times n_j$,
		\begin{align*}
		    \left(E^j_{\alpha\beta}\right)_{\alpha^\prime\beta^\prime} = \begin{cases}
		    \ 1,& \mathrm{if}
		     \ \alpha=\alpha^\prime ~ \text{and} ~ \beta=\beta^\prime\\
		    \ 0,& \mathrm{otherwise}
		    \end{cases},
		\end{align*}
		and let $L^j_{\alpha\beta} = E^j_{\beta\alpha}-E^j_{\alpha\beta}$. Skew-symmetric matrices $L^j_{\alpha\beta}$, $1\leq \alpha<\beta\leq n_j$, form a basis of the tangent space $T_{\mathbb{1}}O(n_j)$ to the orthogonal group at the identity matrix $\mathbb{1}\in O(n_j)$. 
		For $L\in T_{\mathbb{1}}O(n_j)$ let $u\mapsto e^{u L}$ be the one-parameter subgroup of orthogonal matrices such that $e^{0L} = \mathbb{1}$ and $\frac{d}{du}e^{uL} = L e^{uL} =e^{uL} L$ and define a family of matrices in $O(n_j)$ via
		\begin{align}\label{eq: g}
		    g(\mathbf{u}^j)=  \prod_{\alpha = 1}^{r_j}\prod_{\beta=r_j+1}^{n_j} e^{u^j_{\alpha\beta}L^j_{\alpha\beta}},\quad \mathbf{u}^j=(u^j_{\alpha\beta}),
		\end{align}
		where orthogonal matrices $e^{u^j_{\alpha\beta} L^j_{\alpha\beta}}$ in the product are ordered according to the order 
		\begin{align}\label{eq: order}
		    (1,r_j+1), (2,r_j+1),\dots, (r_j,r_j+1),(1,r_j+2),\dots, (r_j,r_j+2), \dots,(1,n_j),\dots,(r_j,n_j)
		\end{align}
		on the set of pairs $(\alpha,\beta)$.
	Note that the associated family $g(\mathbf{u}^j)W_j\subset Gr(r_j,n_j)$ of $r_j$-planes contains a neighborhood of $W_j\in Gr(r_j,n_j)$. This follows from the transitivity of the action of the orthogonal group on the Grassmannian. We will need formulas for partial derivatives of \eqref{eq: g}. By Leibniz's rule we obtain for $1\leq j\leq d$ and $1\leq \lambda\leq r_j$, $r_j+1\leq \mu \leq n_j$
	\small\begin{align*}
	    \frac{\partial g}{\partial u^j_{\lambda\mu}}(\mathbf{u}^j) ~ = ~ \left(\prod_{\alpha=1}^{r_j}\prod_{\beta=r_j+1}^{\mu-1} e^{u^j_{\alpha\beta} L^j_{\alpha\beta}}\right) \left(\prod_{\alpha=1}^{\lambda-1} e^{u^j_{\alpha\mu} L^j_{\alpha\mu}}\right)  e^{u^j_{\lambda\mu}L^j_{\lambda\mu}} L^j_{\lambda\mu} \left(\prod_{\alpha=\lambda+1}^{r_j}  e^{u^j_{\alpha\mu} L^j_{\alpha\mu}}\right)\left( \prod_{\alpha=1}^{r_j}\prod_{\beta=\mu+1}^{n_j} e^{u^j_{\alpha\beta} L^j_{\alpha\beta}}\right).
	\end{align*}\normalsize
	Evaluation of the last expression at $\mathbf{u}^j=\mathbf{0}$ gives
	\begin{align}\label{eq: dg at 0}
	    \frac{\partial g}{\partial u^j_{\lambda\mu}}(\mathbf{0}) = L^j_{\lambda\mu}.
	\end{align}
	Now, the second order partial derivatives of \eqref{eq: g} evaluated at $\mathbf{u}^j= \mathbf{0}$ equal
	\begin{align}\label{eq d2 g at 0}
	    \frac{\partial^2 g}{\partial u^{j}_{\lambda\mu}\partial u^j_{\lambda^\prime\mu^\prime}}(\mathbf{0}) = L^j_{\lambda\mu}L^j_{\lambda^\prime\mu^\prime} = -\delta_{\mu\mu'}E^j_{\lambda \lambda'}-\delta_{\lambda\lambda'}E^j_{\mu\mu'},
	\end{align}
	where $(\lambda,\mu)\leq (\lambda^\prime,\mu^\prime)$ with respect to the above defined order. 
	Next we define a local parametrization of $\mathcal{T}_{d,\mathbf{r}}$ around $T=(t_{i_1\dots i_d})$ by
		\begin{equation}\label{eq: nbhd}
			T(\mathbf{u}^1,\dots,\mathbf{u}^d,S)~ = ~ \sum_{j=1}^d\sum_{i_j=1}^{r_j} (t_{i_1\dots i_d}+s_{i_1\dots i_d}) ~g(\mathbf{u}^1)e^1_{i_1}\otimes \cdots \otimes g(\mathbf{u}^d)e^d_{i_d},
		\end{equation}
		where $\mathbf{u}^j\in \R^{r_j(n_j-r_j)}$, $j\in [d]$, and $S=(s_{i_1\dots i_d})\in \R^{r_1}\otimes \cdots \otimes\R^{r_d}$ together constitute a family of parameters of dimension $\sum_{j=1}^d r_j(n_j-r_j) + \prod_{j=1}^d r_j=\dim(\mathcal{T}_{d,\mathbf{r}})$.

Note that at $\mathbf{0} = (\mathbf{0},\dots,\mathbf{0},\mathbf{0})=(\mathbf{u}^1,\dots,\mathbf{u}^d,S)$ we have $T(\mathbf{0})=T$. 
The first order partial derivatives of \eqref{eq: nbhd} evaluated at $\mathbf{0}$ equal
\begin{align}\label{eq: dTs}
    \frac{\partial T}{\partial s_{i_1\dots i_d}}(\mathbf{0}) &= e^1_{i_1}\otimes \dots\otimes e^d_{i_d}
\end{align}
and 
\begin{equation}\label{eq: dTu}
\begin{aligned}\frac{\partial T}{\partial u^j_{\lambda\mu}}(\mathbf{0}) &= \sum_{k=1}^d \sum_{i_k=1}^{r_k} t_{i_1\dots i_d}\, e^1_{i_1}\otimes \dots\otimes e^{j-1}_{i_{j-1}}\otimes L^j_{\lambda\mu} e^j_{i_j}\otimes e^{j+1}_{i_{j+1}}\otimes \dots\otimes e^d_{i_d}\\
&= \sum_{k \in [d]\setminus j} \sum_{i_k=1}^{r_k} t_{i_1\dots i_{j-1}\lambda\, i_{j+1}\dots i_d}\, e^1_{i_1}\otimes \dots\otimes e^{j-1}_{i_{j-1}}\otimes e^j_{\mu}\otimes e^{j+1}_{i_{j+1}}\otimes \dots\otimes e^d_{i_d},  
\end{aligned}\end{equation}
where in \eqref{eq: dTu} we use \eqref{eq: dg at 0} and the formula 
\begin{align*}
    L^j_{\lambda\mu}e^j_{i_j} = 
    \begin{cases}
    \ e^j_\mu,& \mathrm{if}\ i_j=\lambda\\
    \ -e^j_\lambda,& \mathrm{if}\ i_j = \mu\\
    \ 0,&\mathrm{otherwise}
    \end{cases}.
\end{align*}
Note that $\lambda \in \{1,2,\dots,r_j\}$, $\mu \in \{r_j+1,\dots,n_j\}$, and the sum in \eqref{eq: dTu} in the $i_j$-th index runs over $i_j \in \{1,2,\dots,r_j\}$ so that $i_j = \lambda$ is the only relevant case.
It immediately follows from \eqref{eq: dTs} and \eqref{eq: dTu} that
\begin{align}\label{eq: inner_s_with_s}
			\left\langle \frac{\partial T}{\partial s_{i_1\dots i_d}}(\mathbf{0}), \frac{\partial T}{\partial s_{j_1\dots j_d}}(\mathbf{0})\right\rangle  ~ &= ~ \delta_{i_1j_1}\dots\delta_{i_dj_d}, \\
		\label{eq: scalar b and t}
			\left\langle  \frac{\partial T}{\partial s_{i_1\dots i_d}}(\mathbf{0}), \frac{\partial T}{\partial u^{j}_{\lambda\mu}}(\mathbf{0})\right\rangle ~ &= ~ 0, 
		\\\label{eq: whoknowswhattocallitnotme}
		\left\langle  \frac{\partial T}{\partial u^{j}_{\lambda\mu}}(\mathbf{0}), \frac{\partial T}{\partial u^{j'}_{\lambda'\mu'}}(\mathbf{0})\right\rangle ~ &= ~ \delta_{jj'}\delta_{\mu\mu'}\sum_{k\in[d]\setminus j} \sum_{i_k=1}^{r_k} t_{i_1\dots i_{j-1}\lambda i_{j+1} \dots i_d}\, t_{i_1\dots i_{j-1} \lambda' i_{j+1}\dots i_d},
		\end{align}
for every  $ (i_1,\dots,i_d),(j_1,\dots,j_d)\in [r_1]\times\dots\times [r_d]$, $j,j' \in [d]$, $(\lambda,\mu)\in [r_j]\times \{r_j+1,\dots,n_j\}$ and $(\lambda^\prime,\mu^\prime)\in [r_{j^\prime}]\times \{r_{j^\prime}+1,\dots,n_{j^\prime}\}$.
With these formulas on hand we compute the Gram matrix $G_{T=T(\mathbf{0})}$. The rows and columns of this symmetric $\dim(\mathcal{T}_{d,\mathbf{r}})\times \dim(\mathcal{T}_{d,\mathbf{r}})$ matrix are indexed by the local coordinates in the parametrization \eqref{eq: nbhd}, and each entry of $G_T$ is the inner product of the corresponding partial derivatives of \eqref{eq: nbhd} evaluated at $(\mathbf{u}^1,\dots,\mathbf{u}^d,S)=\mathbf{0}$. First, formulas \eqref{eq: inner_s_with_s}, \eqref{eq: scalar b and t} and \eqref{eq: whoknowswhattocallitnotme} imply that $G_{T}$ has a block diagonal structure, with the $d+1$ many blocks given as 
\begin{equation*}\label{eq: first Gram}
    \renewcommand\arraystretch{2.8}
    G_{T}=
    \begin{pmatrix}       \mbox{\normalfont\large$G_{\frac{\partial T}{\partial \mathbf{u}^1}(\mathbf{0})}$} & \vline & \mbox{\normalfont\Large 0} & \vline & \cdots & \vline & \mbox{\normalfont\Large 0} & \vline & \mbox{\normalfont\Large 0}\\ \hline
    \mbox{\normalfont\Large 0} & \vline & \mbox{\normalfont\large$G_{\frac{\partial T}{\partial \mathbf{u}^2}(\mathbf{0})}$} & \vline & \ddots & \vline & \ddots & \vline & \vdots\\ \hline
    \vdots & \vline & \ddots & \vline & \ddots & \vline & \ddots & \vline & \vdots \\ \hline
    \mbox{\normalfont\Large 0} & \vline & \ddots & \vline & \ddots & \vline & \mbox{\normalfont\large$G_{\frac{\partial T}{\partial \mathbf{u}^d}(\mathbf{0})}$} & \vline & \mbox{\normalfont\Large 0} \\ \hline 
    \mbox{\normalfont\Large 0} & \vline & \cdots & \vline & \cdots & \vline &
    \mbox{\normalfont\Large 0} & \vline &
    \mbox{\normalfont\large$G_{\frac{\partial T}{\partial S}(\mathbf{0})}$}
    \end{pmatrix}.
\end{equation*}
Entries of the block $G_{\frac{\partial T}{\partial \mathbf{u}^j}(\mathbf{0})}$, $j\in [d]$, are inner products \eqref{eq: whoknowswhattocallitnotme} with $j=j'$, whereas the last block $G_{\frac{\partial T}{\partial S}(\mathbf{0})}$ of inner products \eqref{eq: inner_s_with_s} is the $(\prod_{j=1}^d r_j\times \prod_{j=1}^d r_j)$ identity matrix. If we order the rows and columns of the block $G_{\frac{\partial T}{\partial \mathbf{u}^j}(\mathbf{0})}$ according to \eqref{eq: order},
then \eqref{eq: whoknowswhattocallitnotme} with $j=j'$ implies that $G_{\frac{\partial T}{\partial \mathbf{u}^j}(\mathbf{0})}$ has a further block structure,
\begin{equation*}\label{eq: second Gram}
    \renewcommand\arraystretch{2}
    G_{\frac{\partial T}{\partial \mathbf{u}^{j}}(\mathbf{0})} = \begin{pmatrix}
     \hspace{8pt}\mbox{\normalfont\Large $A_j$} & \vline & \mbox{\normalfont\Large 0} & \vline & \cdots & \vline & \mbox{\normalfont\Large 0}\hspace{8pt}\\ \hline
     \mbox{\normalfont\Large 0} & \vline & \mbox{\normalfont\Large $A_j$} & \vline & \ddots & \vline & \vdots \hspace{8pt} \\ \hline
     \vdots & \vline & \ddots & \vline & \ddots & \vline & \mbox{\normalfont\Large 0} \hspace{8pt} \\ \hline 
     \mbox{\normalfont\Large 0} & \vline & \cdots & \vline & \mbox{\normalfont\Large 0} & \vline & \mbox{\normalfont\Large $A_j$} \hspace{8pt}
    \end{pmatrix}.
\end{equation*}
More precisely, $G_{\frac{\partial T}{\partial\mathbf{u}^j}(\mathbf{0})}$ is a block diagonal $(r_j(n_j-r_j)\times r_j(n_j-r_j))$ matrix, with $n_j-r_j$ identical blocks $A_j$. It follows from \eqref{eq: whoknowswhattocallitnotme} with $j=j'$, $\mu=\mu'$  and \eqref{eq:flat} that the $r_j\times r_j$ matrix $A_j$ is the Gram matrix of the rows of the $j$-th flattening of the tensor $T$. Since the $j$-th entry $r_j$ in the multilinear rank $\mathbf{r}=(r_1,\dots,r_j,\dots,r_d)$ of $T\in \mathcal{T}_{d,\mathbf{r}}$ equals the rank of the $j$-th flattening (see Section \ref{section:preliminaries}), the matrix $A_j$ and hence $G_{\frac{\partial T}{\partial\mathbf{u}^j}(\mathbf{0})}$ are nonsingular. Repeating the same argument for every block shows that also $G_{T}$ is nonsingular and hence vectors \eqref{eq: dTs} and \eqref{eq: dTu} form a basis of the tangent space of $\mathcal{T}_{d,\mathbf{r}}$ at $T$.\\

The block diagonal structure of $G_T$ and of $G_{\frac{\partial T}{\partial \mathbf{u}^j} (\mathbf{0})}$, $j\in [d]$, shows that the inverse matrix $G_{T}^{\,-1}$ is also block diagonal, with blocks of the same size. In particular we have
\begin{equation*}\label{eq: inverse zero}
    (G_{T}^{\,-1})_{\alpha\beta} = 0,
\end{equation*}
in the following cases:
\begin{itemize}
    \item[-] $\alpha$ corresponds to a parameter $s_{i_1\dots i_d}$ and $\beta$ corresponds to a parameter $u^j_{\lambda\mu}$.
    \item[-] $\alpha$ corresponds to a parameter $u^{j}_{\lambda\mu}$ and $\beta$ corresponds to a parameter $u^{j'}_{\lambda'\mu'}$, with $j\neq j'$ or with $j=j'$ but $\mu\neq \mu'$.
\end{itemize}
Next we concentrate on the computation of the mean curvature vector at the point $T$. By \eqref{eq: H_formal} we need to show that
\begin{equation}\label{eq: H_T-sum-formula}
	H_T ~ = ~ \mathrm{Tr}\left (G^{-1}_{T}\cdot \left(\mathrm{d}^{\,2} T(\mathbf{0})\right)^\perp\right)  ~ = ~ 0,
\end{equation}
where $\left(\mathrm{d}^2T(\mathbf{0})\right)^\perp$ is the $\dim(\mathcal{T}_{d,\mathbf{r}})\times \dim(\mathcal{T}_{d,\mathbf{r}})$ matrix whose entries are normal components of second order partial derivatives of the parametrization \eqref{eq: nbhd} evaluated at $\mathbf{0}=(\mathbf{u}^1,\dots,\mathbf{u}^d,S)$. Taking into account the above discovered structure of $G_T^{-1}$, we note that the only second order derivatives of \eqref{eq: nbhd} that are relevant for the computation of $H_T$ are
\begin{align*}
\frac{\partial^2 T}{\partial s_{i_1\dots i_d} \partial s_{j_1\dots j_d}} (\mathbf{0})~ =~ 0,\quad (i_1,\dots, i_d),\ (j_1,\dots, j_d)\in [r_1]\times \dots\times [r_d],    
\end{align*}
and 
\begin{align}\label{eq: t and t}
		\dfrac{\partial^2 T}{\partial u^{j}_{\lambda\mu}\partial u^j_{\lambda'\mu}} (\mathbf{0})~ = ~ \sum_{k=1}^d \sum_{i_k=1}^{r_k}t_{i_1\dots i_d} e^1_{i_1}\otimes \cdots \otimes e^{j-1}_{i_{j-1}} \otimes \dfrac{\partial^2 g}{\partial u^j_{\lambda\mu}\partial u^j_{\lambda'\mu} }(\mathbf{0})e^{j}_{i_{j}} \otimes e^{j+1}_{i_{j+1}}\otimes \cdots \otimes e^d_{i_d},
\end{align}
		where	$j\in [d]$, $\lambda,\lambda'\in [r_j]$ and $\mu\in \{r_j+1,\dots, n_j\}$.
Using \eqref{eq d2 g at 0} we obtain 
\begin{align*}
    \frac{\partial^2 g}{\partial u^j_{\lambda\mu}\partial u^j_{\lambda'\mu}}(\mathbf{0})e^j_{i_j} = \left(-\delta_{\lambda\lambda'} E^j_{\mu\mu}-E^j_{\lambda\lambda'}\right) e^j_{i_j} = -\delta_{i_j\lambda'} e^j_\lambda
\end{align*}
and applying it to \eqref{eq: t and t} gives
\begin{align}\label{eq: t and t}
		\dfrac{\partial^2 T}{\partial u^{j}_{\lambda\mu}\partial u^{j}_{\lambda'\mu}}(\mathbf{0}) ~
		=~ -\sum_{k\in[d]\setminus j} \sum_{i_k=1}^{r_k}t_{i_1\dots i_{j-1}\lambda' i_{j+1}\dots i_d} ~e^1_{i_1}\otimes \cdots \otimes e^{j-1}_{i_{j-1}}\otimes e^{j}_{\lambda} \otimes e^{j+1}_{i_{j+1}}\otimes \cdots \otimes e^d_{i_d}.\nonumber
		\end{align}
Although this last computation reveals a nonzero vector we observe that all the indices $(i_1,\dots,i_{j-1},\lambda,i_{j+1},\dots,i_d)$ appearing in its expression range in $[r_1]\times\cdots\times [r_j]\times\cdots\times [r_d]$. This implies that \eqref{eq: t and t} lies in the tangent space $T_{T}\mathcal{T}_{d,\mathbf{r}}$, and hence its normal component $\left(	\dfrac{\partial^2 T}{\partial u^{j}_{\lambda\mu}\partial u^{j}_{\lambda'\mu}}(\mathbf{0}) \right)^\perp$ is zero. 
Thus, \eqref{eq: H_T-sum-formula} holds and this completes the proof.

\section{Linear optimization over the Segre variety}\label{sec:Linear}

In this section we consider the Segre variety, i.e., the manifold $\mathcal{S}_d=\mathcal{T}_{d,(1,\dots,1)}$ of tensors of order $d$ and multilinear rank $(1,\dots,1)$. Our goal is to prove \Cref{Theorem:Linear}. Let $T\in \mathcal{S}_d$ and let $P_{\vec{\ell}}$ be any affine hyperplane containing the tangent space to $\mathcal{S}_d$ at $T$. We will show that in any neighborhood of $T\in \mathcal{S}_d$ there are rank-one tensors lying on both sides of $P_{\vec{\ell}}$. To do so, we construct two curves $\gamma^+, \gamma^-:(-\varepsilon,\varepsilon)\to \mathcal{S}_d$ passing through $T=\gamma^+(0)=\gamma^-(0)$ and such that $\pm\langle \gamma^\pm(u)-T,\vec{\ell}\rangle >0$ for $0<u<\varepsilon$. 
The last condition simply means that the rank-one tensors $\gamma^+(u)$ and $\gamma^-(u)$ lie in the opposite open half-spaces formed by the affine hyperplane $P_{\vec{\ell}}$. Note that the restriction of the action \eqref{eq:action} of $O(\mathbf{n})=O(n_1)\times \dots\times O(n_d)$ to $\mathcal{S}_d$ is transitive. This together with the fact that \eqref{eq:action} preserves the inner product \eqref{Frob} imply that it is enough to prove the statement when $T=e^1_1\otimes \dots\otimes e_1^d\in \mathcal{S}_d$. 

Before we proceed with the construction of curves $\gamma^+$ and $\gamma^-$, let us investigate the structure of the normal space $N_T \mathcal{S}_d$ to the Segre variety $\mathcal{S}_d$ at $T=e^1_1\otimes \dots\otimes e_1^d\in \mathcal{S}_d$. We have an orthogonal direct sum decomposition, which we denote by $\oplus^{\perp}$:
\begin{equation}\label{eq:dec_normal}
    N_T \mathcal{S}_d = N_2 \oplus^{\perp} N_3 \oplus^{\perp} \cdots \oplus^{\perp} N_d,
\end{equation}
where
\begin{equation*}
    N_k = \text{Span} \, \{ e_1^1 \otimes \cdots \otimes w_{i_1} \otimes \cdots \otimes w_{i_k} \otimes \cdots \otimes e_1^d \,\, : \,\, 1 \leq i_1 < \cdots < i_k \leq d, \,\, w_{i_j} \in (e_1^{i_j})^{\perp} \}.
\end{equation*}
With these notations we may also write
\begin{equation*}
    N_0 = \text{Span} \, \{ e_1^1 \otimes \cdots \otimes e_1^d \} = \text{Span} \, \{ T \}
\end{equation*}
and
\begin{equation*}
    N_1 = \text{Span} \, \{ e_1^1 \otimes \cdots \otimes e_1^{i-1} \otimes w_{i} \otimes e_1^{i+1} \otimes \cdots \otimes e_1^d \,\, : \,\, 1 \leq i \leq d, \,\, w_{i} \in (e_1^{i})^{\perp} \},
\end{equation*}
where $N_0 + N_1 = T_T \mathcal{S}_d$ is the tangent space to $\mathcal{S}_d$ at $T$. It is straightforward to check that the subspaces $N_k$, $k=0,\dots,d$, are pairwise orthogonal with respect to the inner product \eqref{Frob}. 

By assumption the vector $\vec{\ell}$ belongs to the normal space $N_T\mathcal{S}_d$. 
Let $k\geq 2$ be the smallest natural number such that $\vec{\ell}$ has nonzero component in $N_k$. For example, if $k=2$, the vector $\vec{\ell}$ has nonzero coefficient in front of one of the basic tensors $e_1^1 \otimes \cdots \otimes e_{\alpha_1}^{i_1} \otimes \cdots \otimes e_{\alpha_2}^{i_2} \otimes \cdots \otimes e_1$ for $\alpha_1, \alpha_2 > 1$ and some indices $1\leq i_1<i_2\leq d$. We first prove an auxiliary lemma.


\begin{lemma}\label{lemma: curve}
Let $\vec{\ell}$ have nonzero coefficient $c(\vec{\ell},\alpha_1,\dots,\alpha_k)\neq 0$ in front of the basic tensor $e^1_{\alpha_1} \otimes \cdots \otimes e^k_{\alpha_k} \otimes e^{k+1}_1 \otimes \cdots \otimes e^d_1 \in N_k$, where $\alpha_1,\dots,\alpha_k>1$. Then the curve 
\begin{align}\label{eq:curve}
    \gamma = \{ \,\, ( e^{u L_{1,\alpha_1}^1} e^1_1 ) \otimes \cdots \otimes ( e^{u L_{1,\alpha_k}^k} e^k_1 ) \otimes e^{k+1}_1 \otimes \cdots \otimes e^d_1 \,\, : \,\, u \in (-\varepsilon, \varepsilon) \,\, \} \subset \mathcal{S}_d
\end{align} 
satisfies $\langle \gamma^{(j)}(0), \vec{\ell} \rangle = 0$ for all $j \in \{0,1,\dots,k-1\}$ and $\langle \gamma^{(k)}(0), \vec{\ell} \rangle = k! \, c(\vec{\ell}, \alpha_1,\dots,\alpha_k )$.
\end{lemma}

\begin{proof}
This follows from Leibniz' rule. Indeed we have that
\begin{align*}
    \gamma(0) &= e^1_1 \otimes \cdots \otimes e^d_1=T \in N_0\\
    \gamma'(0) &= \sum_{j=1}^k e^1_1 \otimes \cdots \otimes e^j_{\alpha_j} \otimes \cdots \otimes e^d_1 \in N_1 \\
     & \vdots \\
    \gamma^{(k)}(0) &= W + \,\, k! \, e^1_{\alpha_1} \otimes \cdots \otimes e^k_{\alpha_k} \otimes e^{k+1}_1 \otimes \cdots \otimes e^d_1,
\end{align*}
with $W\in N_0 \oplus^{\perp} N_1 \oplus^{\perp} \cdots \oplus^{\perp} N_{k-1}$. The condition $\vec{\ell}\in N_k$,  orthogonality of $N_0, N_1,\dots, N_{k-1}, N_k$ and  orthogonality of basic tensors imply the claim.
\end{proof}

We now proof the main result of this section.\\

\noindent
\emph{Proof of \Cref{Theorem:Linear}:}
Let $\vec{\ell}\in N_k\oplus^\perp\dots\oplus^\perp N_d$, $k\geq 2$, and $\gamma: (-\varepsilon,\varepsilon)\rightarrow \R$ be as in Lemma \ref{lemma: curve}.
The Taylor expansion
\begin{align*}
    \gamma(u) = \gamma(0) + \gamma'(0) u + \gamma''(0) \frac{u^2}{2!} + \cdots + \gamma^{(k-1)}(0) \frac{u^{k-1}}{(k-1)!} + \gamma^{(k)}(0) \frac{u^k}{k!} + \cdots
\end{align*}
and Lemma \ref{lemma: curve} imply that the (local) position of $\gamma(u)$ (for small $u>0$) with respect to $P_{\vec{\ell}}=\{T+V\,:\,\langle \vec{\ell},V\rangle = 0\}$ is determined by the sign of $\langle \gamma^{(k)}(0), \vec{\ell} \rangle = k! \, c(\vec{\ell}, \alpha_1,\dots,\alpha_k)$.
If $c(\vec{\ell},\alpha_1,\dots,\alpha_k)>0$ define $\gamma^+$ to be $\gamma$, while  if $c(\vec{\ell},\alpha_1,\dots,\alpha_k)<0$ set $\gamma^-=\gamma$.
The curve $\gamma^-$ (respectively $\gamma^+$) is then defined to be 
\begin{align*}
   \tilde{\gamma}\ =\ \{ \,\, ( e^{-u L_{1,\alpha_1}^1} e^1_1 ) \otimes ( e^{u L_{1,\alpha_2}^2} e^2_1) \otimes  \cdots \otimes ( e^{u L_{1,\alpha_k}^k} e^k_1 ) \otimes e^{k+1}_1 \otimes \cdots \otimes e^d_1\,:\,  u\in (-\varepsilon,\varepsilon) \,\, \}\subset \mathcal{S}_d.
\end{align*}
Note that the single difference is in the sign in front of the parameter $u$ in the first tensor factor. This change in the sign is needed to flip the sign of the inner product $\langle \gamma^{(k)}(0),\vec{\ell}\rangle$ as $\langle \tilde\gamma^{(j)}(0),\vec{\ell}\rangle =0$ for $j=0,1,\dots,k-1$ and $\langle \tilde\gamma^{(k)}(0),\vec{\ell}\rangle =-k!\,c(\vec{\ell},\alpha_1,\dots,\alpha_k)$. This computation follows exactly the same path as the proof of \Cref{lemma: curve}.
In both cases it follows from the definition of $\gamma^\pm$ and Taylor expansion that $\pm\langle \gamma^\pm(u)-T,\vec{\ell}\rangle >0$ for small enough $u>0$.

In general, if $\vec{\ell}\in N_k\oplus^\perp\dots\oplus^\perp N_d$ has a non-zero coefficient in front of the basic tensor $e^1_1\otimes \dots\otimes e^{i_1}_{\alpha_1}\otimes \dots\otimes e^{i_k}_{\alpha_k} \otimes \dots\otimes e^d_{1}\in N_k$ with $1\leq i_1<\dots<i_k\leq d$ and $\alpha_1,\dots,\alpha_k>1$, one has to consider the curve
\begin{align*}
      \gamma = \{ \,\, e^1_1\otimes \dots\otimes ( e^{u L_{1,\alpha_1}^{i_1}} e^{i_1}_1 ) \otimes \cdots \otimes ( e^{u L_{1,\alpha_k}^{i_k}} e^{i_k}_1 ) \otimes \dots\otimes  e^{d}_1 \,\, : \,\, u \in (-\varepsilon, \varepsilon) \,\, \} \subset \mathcal{S}_d
\end{align*}
instead of \eqref{eq:curve}. All other steps are identical. 

The above proves that there are tensors in an arbitrarily small neighborhood of $T\in \mathcal{S}_d$ that lie on either side of the affine hyperplane $P_{\vec{\ell}}$. In particular, the restriction of a linear functional to $\mathcal{S}_d$ attains no local minima or maxima.
\qed\\

\begin{remark}\label{rem:non-degenerate}
Let $M\subseteq \R^n$ be a minimal submanifold such that for every point $x\in M$ and every unit normal vector $\vec{\ell}\in N_xM$ the second fundamental form 
\begin{align}\label{eq:bl}
    u, v\in T_xM \rightarrow \langle b(u,v),\vec{\ell}\rangle
\end{align}
with respect to $\vec{\ell}$ is non-degenerate. 
Because of minimality of $M$, \eqref{eq:bl} has to have a negative and a positive eigenvalue. Geometrically this means that no small neighborhood of $x\in M$ is completely contained in one of the half-spaces $P^\pm_{\vec{\ell}}=\{x+v\,:\,\pm\langle \vec{\ell},v\rangle\geq 0\}$. In particular, no linear functional $L_{\vec{\ell}}(x)=\langle \vec{\ell},x\rangle$ attains a local minimum or a local maximum on $M$.
Note however that the Segre variety $\mathcal{S}_d$ does not satisfy the above condition on the second fundamental form. 
Indeed, one can see that all second order derivatives of the parametrization \eqref{eq: nbhd} evaluated at $T=T(\mathbf{0})$ lie in $N_0\oplus^{\perp} N_1\oplus^{\perp} N_2$. Thus, for all $\vec{\ell}\in N_k$ with $k>2$ (see the decomposition \eqref{eq:dec_normal}) the bilinear form \eqref{eq:bl} is trivial. 
That is why, in Theorem \ref{Theorem:Linear} we have to use a neater argument to prove nonexistence of local extrema of linear functionals on $\mathcal{S}_d$.
\end{remark}

We now derive Corollary \ref{cor:A}, which asserts that (when nonconstant) the restriction of a linear functional to the affine-linear slice $\mathcal{S}_d\cap A$ of the Segre variety has no local minima or maxima as well.\\


\noindent
\emph{Proof of \Cref{cor:A}:} Let $A$ be the affine hyperplane $\{T\in \R^{n_1}\otimes\dots\otimes\R^{n_d}\,:\, \langle T,\vec{a}\rangle = c\}$. Observe first that the Segre variety $\mathcal{S}_d$ is \emph{conical}, i.e., $\lambda T \in \mathcal{S}_d$ for any $T\in \mathcal{S}_d$ and any $\lambda\in\mathbb{R}\setminus\{0\}$. Assume that $T\in \mathcal{S}_d\cap A$ is a local minimum (maximum) for the restriction of a linear functional $L_{\vec{\ell}}(V)=\langle\vec{\ell},V\rangle$ on $\mathcal{S}_d\cap A$.  
Since the restriction of $L_{\vec{\ell}}$ to $A$ is assumed to be nonconstant, we have in particular that $\vec{\ell}$ is not proportional to $ \vec{a}$. Then $\vec{\ell},\vec{a}$ span a two-dimensional subspace, which must intersect any hyperplane. Therefore there exists $\mu \in \mathbb{R}$ with $\vec{v} = \vec{\ell} + \mu \vec{a}$ orthogonal to $T$. If $A$ is a linear hyperplane ($c=0$), then $\langle \vec{\ell}, T\rangle =0$ holds automatically because in this case $\mathcal{S}_d\cap A$ remains conical and $T\in T_T(\mathcal{S}_d\cap A)$. Note that the coefficient of $\vec{\ell}$ can be assumed equal to $1$ since any rescaling of $\vec{\ell}$ would produce another linear functional which also has $T$ as a local minimum (maximum) on $\mathcal{S}_d \cap A$. 
The condition $L_{\vec{v}}(T)=0$ implies that $T\in\mathcal{S}_d\cap A$ is a local minimum (maximum) of $L_{\vec{v}}$ restricted to $\mathcal{S}_d$. This contradicts to \Cref{Theorem:Linear}.
\qed\\


\begin{figure}[h]
    \centering
    \includegraphics[width=0.3\textwidth]{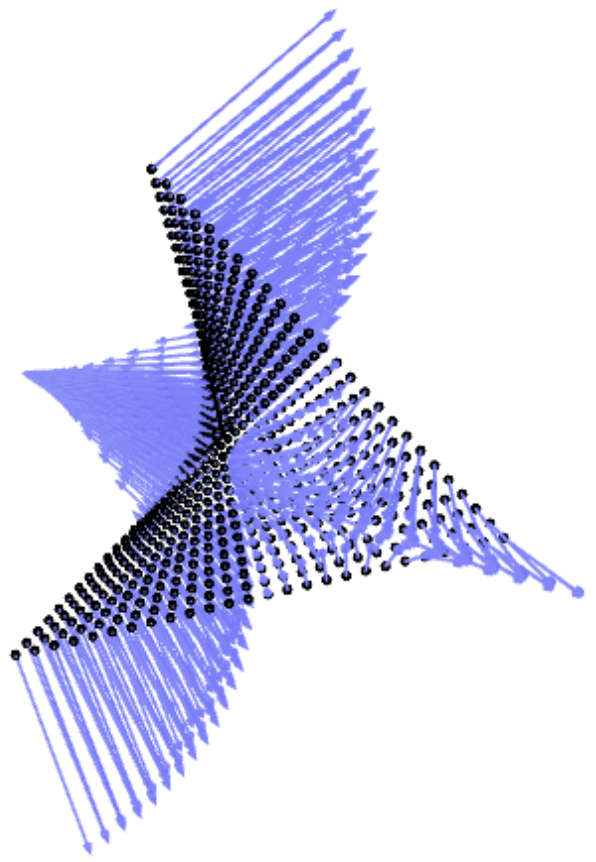}
    
    
    
    \caption{Mean curvature vector field on $\mathcal{M}$, the independence model for two binary discrete random variables.}
    \label{fig:vector-field}
\end{figure}
Observe that, while \Cref{cor:A} shows that the absence of local minima and maxima is preserved under affine hyperplane sections, the same is not true for minimality. \Cref{fig:vector-field} depicts the mean curvature field (blue vectors) for (the black pointed surface)
\begin{align*}
    \mathcal{S}_2\cap A,\quad A\ =\ \{ S=(s_{ij})_{i,j=1,2}\in \R^2\otimes \R^2\,:\, s_{11}+s_{12}+s_{21}+s_{22}=1\}.
\end{align*} 
Since this vector field is evidently nonzero, this surface $\mathcal{S}_2 \cap A$ is not minimal. If we restrict ourselves to tensors with nonnegative entries we obtain a statistical independence model, as explained in the introduction. This is precisely the portion of $\mathcal{S}_2\cap A$ captured in \Cref{fig:vector-field}. We make the code to compute the mean curvature for this example available at 
\begin{center}
    \url{https://mathrepo.mis.mpg.de/tensorminimality}.
\end{center}

\bibliographystyle{plain}
\bibliography{references}
\vspace{1cm}
\end{document}